\documentclass{amsart}
\usepackage{mathtools, microtype, mathrsfs, xfrac, hyperref, amssymb, enumerate, xspace, todonotes}
\usepackage[latin1]{inputenc}
\usepackage{tikz-cd}

\numberwithin{equation}{section}

\numberwithin{subsection}{section}

\allowdisplaybreaks[1]

\newtheorem*{namedtheorem}{\theoremname}
\newcommand{\theoremname}{testing}
\newenvironment{named}[1]{\renewcommand\theoremname{#1}
\begin{namedtheorem}}
{\end{namedtheorem}}

\newtheorem{theorem}{Theorem}[section]
\newtheorem{proposition}[theorem]{Proposition}
\newtheorem{proposition-definition}[theorem]
{Proposition-Definition}
\newtheorem{corollary}[theorem]{Corollary}

\theoremstyle{definition}

\newtheorem{remark}[theorem]{Remark}

\theoremstyle{remark}

\renewcommand{\mathcal}{\mathscr}

\newcommand\cA{\mathcal{A}} \newcommand\cB{\mathcal{B}}

 \newcommand\cP{\mathcal{P}}

 \newcommand\cV{\mathcal{V}}
 \newcommand\cX{\mathcal{X}}

\renewcommand\AA{\mathbb{A}}

\newcommand\GG{\mathbb{G}} 
 
 \newcommand\LL{\mathbb{L}}
 \newcommand\NN{\mathbb{N}}

 \newcommand\ZZ{\mathbb{Z}}

\newcommand\rK{\mathrm{K}} 
 
\newcommand\rO{\mathrm{O}}

\newcommand\rmm{\mathrm{m}}

\newcommand\arr{\ifinner\to\else\longrightarrow\fi}

\newcommand\arrto{\ifinner\mapsto\else\longmapsto\fi}

\newcommand{\eqdef}{\mathrel{\smash{\overset{\mathrm{\scriptscriptstyle def}} =}}}

\renewcommand\th{^\text{th}}

\def\displaytimes_#1{\mathrel{\mathop{\times}\limits_{#1}}}

\def\displayotimes_#1{\mathrel{\mathop{\bigotimes}\limits_{#1}}}

\newcommand\spec{\operatorname{Spec}}

\newcommand\generate[1]{\langle #1 \rangle}

\newlength{\ignora}

\renewcommand{\setminus}{\smallsetminus}

\newcommand{\mmu}{\boldsymbol{\mu}}

\newcommand{\gm}{\GG_{\rmm}}

\newcommand{\GL}{\mathrm{GL}}
\newcommand{\SL}{\mathrm{SL}}
\newcommand{\PGL}{\mathrm{PGL}}

\DeclareFontFamily{U}{mathx}{\hyphenchar\font45}
\DeclareFontShape{U}{mathx}{m}{n}{
      <5> <6> <7> <8> <9> <10>
      <10.95> <12> <14.4> <17.28> <20.74> <24.88>
      mathx10
      }{}
\DeclareSymbolFont{mathx}{U}{mathx}{m}{n}
\DeclareFontSubstitution{U}{mathx}{m}{n}
\DeclareMathAccent{\widecheck}{0}{mathx}{"71}
\DeclareMathAccent{\wideparen}{0}{mathx}{"75}

\renewcommand{\epsilon}{\varepsilon}

\renewcommand{\O}{\rO}

\newcommand{\SO}{\mathrm{SO}}

\newcommand{\sym}{\operatorname{Sym}}

\newcommand{\kstack}{\rK_{0}(\mathrm\overline{{\rm Stack}_{k}})}

\newcommand{\kvar}{\rK_{0}(\mathrm\overline{{\rm Var}_{k}})}

\newcommand{\kcomp}{\widehat{\rK}_{0}(\mathrm\overline{{\rm Var}_{k}})}

\begin{document}

\title[]{The motivic class of the classifying stack\\of the special orthogonal group}

\author[Mattia Talpo]{Mattia Talpo}

\author[Angelo Vistoli]{Angelo Vistoli$^\dagger$}

\address[Talpo]{Department of Mathematics\\
Simon Fraser University\\
8888 University Drive\\
Burnaby BC\\
V5A 1S6 Canada}

\email[Talpo]{mtalpo@sfu.ca}

\address[Vistoli]{Scuola Normale Superiore\\Piazza dei Cavalieri 7\\
56126 Pisa\\ Italy}

\email[Vistoli]{angelo.vistoli@sns.it}


\thanks{$^\dagger$Partially supported by research funds from the Scuola Normale Superiore}

\maketitle


\begin{abstract}
We compute the class of the classifying stack of the special orthogonal group in the Grothendieck ring of stacks, and check that it is equal to the multiplicative inverse of the class of the group.
\end{abstract}

\section{Introduction}

Let $k$ be a field. The Grothendieck ring of varieties $\kvar$ was first defined by Grothendieck in 1964 in a letter to Serre.  Its main application so far is Kontsevich's theory of motivic integration: see for example \cite{looijenga-motivic}.

Variants of this, that contain classes for all algebraic stacks of finite type over $k$ with affine stabilizers, have been introduced by several authors: see \cite{behrend-dhillon}, \cite{ekedahl-grothendieck}, \cite{joyce-motivic-invariant}, \cite{toen-grothendieck-stacks}. In the present paper we use the version due to Ekedahl, which we denote by $\kstack$; it has the merit of being universal, so it maps to all the other versions.

By definition, every algebraic stack $\cX$ of finite type over $k$ with affine stabilizers has a class $\{\cX\}$ in $\kstack$. In particular, given an affine group scheme of finite type $G$ over $k$, we obtain a class $\{\cB G\}$ for the classifying stack $\cB G$ in $\kstack$. The problem of computing $\{\cB G\}$ is very interesting; it is morally related with the problem of stable rationality of fields of invariants, although no direct implication is known (see the discussion in \cite[\S~6]{ekedahl-finite-group}).

The case of a finite group is thoroughly discussed in \cite{ekedahl-finite-group}; in many cases $\{\cB G\} = 1$, although there are examples of finite groups for which $\{\cB G\}\neq 1$ over every field $k$ \cite[Corollary 5.2]{ekedahl-finite-group}.

The case when $G$ is connected is also very interesting. Recall that an algebraic group is \emph{special} if every $G$-torsor is Zariski-locally trivial; $\GL_{n}$, $\SL_{n}$ and $\mathrm{Sp}_{n}$ are all special. If $P \arr S$ is a $G$-torsor and $G$ is special, then we have $\{P\} = \{G\}\{S\}$ (this is immediate when $S$ is a scheme, and it was shown by Ekedahl when $S$ is an algebraic stack). In particular, applying this to the universal torsor $\spec k \arr \cB G$ we get the formula $\{\cB G\} = \{G\}^{-1}$ for special groups.
 
The cases of non-special connected groups $G$ for which $\{\cB G\}$ has been computed include $\PGL_{2}$, $\PGL_{3}$ (by D.~Bergh in \cite{bergh-motivic-classes}) and $\SO_{n}$ when $n$ is odd (by A.~Dhillon and M.~Young in \cite{dhillon-young-SOn}). In all these cases the equality $\{\cB G\} = \{G\}^{-1}$ continues to hold. This is quite surprising, in view of the fact that if $G$ is a reductive non-special group and the characteristic of $k$ is $0$, there exists a $G$-torsor $P \arr S$ such that $\{P\} \neq \{G\}\{S\}$ \cite[Theorem~2.2]{ekedahl-approximation}. It might be related with the fact that quotient spaces of representations of connected algebraic groups tend to be stably rational; in fact, no examples are known in which they are not rational (see \cite{bohing-rationality} for a survey of the known results in this direction). Of course, since, as we say above, no direct implication is known to hold between the two problems, this is pure speculation on our part.

Our lack of insight into why the formula $\{\cB G\} = \{G\}^{-1}$ does hold is revealed by the fact that when it has been proved, it has been by independently computing the two sides. 

In this paper we continue in this line of research, and we compute the class $\{\cB \SO_{n}\}$ for all $n$.

\begin{named}{Theorem}
Assume that the characteristic of $k$ is different from $2$. Let $q$ be a non-degenerate split quadratic form on an $n$-dimensional $k$-vector space. 
Then
   \[
   \{\cB\SO(q)\} = \{\SO(q)\}^{-1}\,.
   \]
\end{named}

Once again, our result is obtained by explicitly computing $\{\cB\SO(q)\}$ (Theorem~\ref{thm:main}), and then comparing what we get with the formula for $\{\SO(q)\}$ that one obtains from \cite{behrend-dhillon}.

Our approach is different from that of A.~Dhillon and M.~Young in \cite{dhillon-young-SOn}, and also gives  an independent proof of their result. Instead of the stratification of the space $\GL_{n}/\O_{n}$ of non-degenerate quadratic forms that they use, we exploit a simpler stratification of the tautological representation of $\SO_{n}$ already used in \cite{molina-vistoli-classifying}.

In subsequent work \cite{spin}, R. Pirisi and the first author have computed the class of $\cB G$ for the exceptional group $G_2$, and the spin groups $\mathrm{Spin}_n$ for $n=7,8$, finding that the formula $\{\cB G\}=\{G\}^{-1}$ holds for these groups as well.

\subsection*{Acknowledgements} This paper has been inspired by conversations of the first author with A.~Dhillon and Z.~Reichstein, whom we thank warmly. We also thank M. Young and the anonymous referee for useful comments.

\section{The Grothendieck ring of algebraic stacks}\label{sec:Grothendieck-ring}

Recall that the Grothendieck ring of algebraic varieties is generated by classes $\{X\}$ of schemes of finite type over $k$, with the ``scissor'' relation $\{X\} = \{Y\} + \{X \setminus Y\}$ for any closed subscheme $Y \subseteq X$ (see for example \cite{looijenga-motivic}). The sum is given by taking disjoint unions, the product by the cartesian product of schemes. The \emph{Lefschetz motive} $\LL \eqdef \{\AA^{1}\}$ is particularly important. 

The localization ring $\kvar[\LL^{-1}]$ has a natural filtration, whose $m\th$ piece is generated by classes of the form $\{X\}\LL^{n}$, where $n \in \ZZ$ and $X$ is a scheme with $\dim X + n \leq -m$. The completion is denoted by $\kcomp$.

A variant of this is due to Ekedahl (\cite{ekedahl-grothendieck}): one takes classes $\{\cX\}$ of algebraic stacks of finite type with affine stabilizers, subject to the scissor relations, and the relation $\{\cV\} = \LL^{r}\{\cX\}$ whenever $\cV \arr \cX$ is a vector bundle of rank~$r$ (for schemes, vector bundles are locally trivial in the Zariski topology, so this relation is a consequence of the scissor relations, but this is definitely false for stacks). Ekedahl shows the remarkable fact that $\kstack$ is the localization of $\kvar$ obtained by inverting $\LL$ and all elements of the form $\LL^{n} - 1$ for $n\in \NN$; as a consequence, the natural map $\kvar \arr \kcomp$ factors through $\kstack$. (The fact that one can define classes in $\kcomp$ for algebraic stacks of finite type had been earlier shown by K.~Behrend and A.~Dhillon in \cite{behrend-dhillon}).

Next we will prove some easy results that will be used in the rest of the paper.

\begin{proposition}\label{prop:affine-bundle}
Let $\cX$ an algebraic stack of finite type over $k$ with affine stabilizers, $\cA \arr \cX$ an affine bundle of relative dimension $d$. Then we have $\{\cA\} = \LL^{d}\{\cX\}$ in $\kstack$.
\end{proposition}

\begin{proof}
The structure group of $\cA$ is, by definition, the semidirect product $\GL_{d} \ltimes \AA^{d}$, which is a special group. If $\cP \arr \cX$ is the principal $\GL_{d} \ltimes \AA^{d}$-bundle associated with $\cA \arr \cX$, then $\{\cP\} = \{\GL_{d} \ltimes \AA^{d}\}\{\cX\} = \{\GL_{d}\}\LL^{d}\{\cX\}$ (\cite[Proposition 1.4~i)]{ekedahl-grothendieck}). On the other hand $\cA$ is the quotient $\cP/\GL_{d}$, so $\cP$ is a $\GL_{d}$-torsor over $\cA$, hence $\{\cP\} = \{\GL_{d}\}\cA$, and the result follows.
\end{proof}

\begin{proposition}\label{prop:semidirect}
Let $G$ be an affine algebraic group over $k$ acting linearly on a $d$-dimensional vector space $V$, considered as a group scheme via addition. Then we have
   \[
   \{\cB (G \ltimes V)\} = \LL^{-d} \{\cB G\}\,.
   \]
\end{proposition}

\begin{proof}
The group $G$ acts on $V = \AA^{d}$ by definition, while $V$ acts on itself by translation. These two actions combine to give an action of $G \ltimes V$ on $V$, which factors through the group of affine transformations $\GL(V)\ltimes V$. The action of $G\ltimes V$ on $V$ is transitive, and the stabilizer of the origin is $G$, hence $[V/(G \ltimes V)] \simeq \cB G$. On the other hand we can consider $[V/(G \ltimes V)]$ as an affine bundle on $\cB (G \ltimes V)$, so the result follows from Proposition~\ref{prop:affine-bundle}.
\end{proof}

\section{The computation}  
Let $k$ be a field of characteristic different from $2$, and $q$ be a non-degenerate quadratic form on an $n$-dimensional $k$-vector space, $\O(q)$ the corresponding orthogonal group over $k$, and $\SO(q) \subseteq \O(q)$ the connected component of the identity.

\begin{theorem}\label{thm:main}
In $\kstack$ we have the equality
   \[
   \{\cB\O(q)\} = 
   \begin{cases}\displaystyle
   \LL^{-m^{2} + 2m}\prod_{i = 1}^{m}(\LL^{2i} - 1)^{-1} &\text{if }\/n = 2m
   \vspace{3pt}\\
   \displaystyle
   \LL^{-m^{2}} \prod_{i = 1}^{m}(\LL^{2i} - 1)^{-1}&\text{if }\/n = 2m + 1\, .
   \end{cases}
   \]
Furthermore, 
  \[
   \{\cB\SO(q)\} = \{\cB\O(q)\}
   \]
if $n$ is odd, while  
  \[
   \{\cB\SO(q)\} = (1+\LL^{-m})\{\cB\O(q)\}
   \]
if $n=2m$ and $q$ is split.
\end{theorem}

The formulas for $\{\cB\O(q)\}$, and $\{\cB\SO(q)\}$ when $n$ is odd, are contained in \cite{dhillon-young-SOn}.

\begin{proof}
Let us assume right away that $q$ is split and pick a basis for $V\simeq k^n$ that puts  $q$ in the standard form
   \[
   q_{n}(x_{1}, \dots, x_{n}) =
   x_1x_{2} + x_{3} x_{4}\dots + x_{2m-1}x_{2m}
   \]
when $n = 2m$, and
   \[
   q_{n}(x_{1}, \dots,
   x_{n}) = x_1x_{2} + x_{3} x_{4}\dots + x_{2m-1}x_{2m}+ x_{2m+1}^{2}
   \]
when $n = 2m+1$. We will denote by $\O_{n}$ the algebraic group of linear transformations preserving this quadratic form, by $h_{n}\colon V \times V \arr k$ the corresponding symmetric bilinear form, and by $\SO_{n}$ the corresponding special orthogonal group.

The arguments that follow will also compute $\{\cB\O(q)\}$ for any non-degenerate quadratic form $q$, because if $q$ and $q'$ are non-degenerate quadratic forms on $V$ then we have an equivalence of stacks $\cB\O(q) \simeq \cB\O(q')$ (see \cite[Remark 4.2]{molina-vistoli-classifying}). Moreover, note that regardless of the fact that $q$ is split or not, if $n$ is odd we have $\O(q)\simeq \mmu_2\times \SO(q)$, and hence $\{\cB\O(q)\}=\{\cB\SO(q)\}$. If $n$ is even, we do need to assume that $q$ is split, in order to carry out our computation of $\{\cB\SO(q)\}$.

Let us proceed by induction. We will include the case $n = 0$ in the formula for $\{\cB\O_{n}\}$; we take $\O_{0}$ to be the trivial group, so that the formula holds in this case. 

For $n = 1$ we have $\SO_{1} = \{1\}$, while $\O_{1} = \mmu_{2}$, so that $\{\cB\O_{1}\} = \{\cB\SO_{1}\} = 1$, and the theorem is verified in this case. So we can assume $n \geq 2$.

We also have $\SO_{2} = \gm$; in this case $\{\cB\SO_{2}\} = (\LL-1)^{-1}$, and Theorem~\ref{thm:main} is verified. In the rest of the proof we will exclude this case.

We will denote $e_{1}$, \dots,~$e_{n}$ the standard basis of $V$. We will identify $V$ with the corresponding affine space $\spec (\sym_{k}V^{\vee})$ over $k$, and denote by $V^{0}$ the complement of the origin in $V$. We set $C$ the closed subscheme of $V^{0}$ defined by the vanishing of $q_{n}$, and $B \eqdef V^{0} \setminus C$. 
The subschemes $C$, $B$ 
of $V^{0}$ are invariant by the natural action of $\O_{n}$. 

In order to compute $\{\cB\O_{n}\}$ and $\{\cB\SO_{n}\}$, notice that if we denote by $G_{n}$ either $\O_{n}$ or $\SO_{n}$ we have
   \[
   \LL^{n}\{\cB G_{n}\} = \{[V/G_{n}]\} = \{[V^{0}/G_{n}]\} + \{\cB G_{n}\}
   \]
so that
   \[
   \{\cB G_{n}\} = (\LL^{n} - 1)^{-1}\{[V^{0}/G_{n}]\}\,.
   \]
On the other hand
   \[
   \{[V^{0}/G_{n}]\} = \{[C/G_{n}]\} + \{[B/G_{n}]\}\,.
   \]
The hypothesis that $n \geq 2$ and $G_{n} \neq \SO_{2}$ ensures that the action of $G_{n}$ on $C$ is transitive. Split $V$ as $\generate{e_{1}, e_{2}} \oplus W$, where $W \eqdef \generate{e_{1}, e_{2}}^{\perp}$; then $\O(W) = \O_{n-2}$, and a simple calculation shows that the stabilizer of $e_{1}$ for the action of $G_{n}$ on $C$ is of the form $G_{n-2} \ltimes W$. The action of a vector $v \in W$ is defined as leaving $e_{1}$ fixed, sending $e_{2}$ to $-q_{n}(v)e_{1} + e_{2} + v$, and $x \in W$ to $x - 2 h_{n}(x, v)e_1$ (recall that $h_n$ is the symmetric bilinear form associated with $q_n$, i.e. $h_n(v,w)=\frac{1}{2}(q_n(v+w)-q_n(v)-q_n(w))$). Thus
   \[
   [C/G_{n}] = \cB(G_{n-2} \ltimes W)\,.
   \]
From Proposition~\ref{prop:semidirect} we obtain
   \[
   \{[C/G_{n}]\} = \LL^{-n+2}\{\cB G_{n-2}\}\,.
   \]

Let us compute $\{[B/G_{n}]\}$; here the action of $G_{n}$ on $B$ is not transitive, and we need a more elaborate construction, taken from \cite[Section 4]{molina-vistoli-classifying}.

Call $Q$ the closed subscheme of $V^{0}$ defined by $q_{n}(x) = 1$. This is also invariant for the action of $\O_{n}$. We have a natural double cover $\gm \times Q \arr B$ defined by $(t, x) \arrto tx$; there is also a free action of $\mmu_{2}$ on $\gm \times Q$ defined by $\alpha(t,x) = (\alpha t, \alpha x)$, and $B = (\gm \times Q)/\mmu_{2}$.

If we let $G_{n}$ act on $\gm \times Q$ by acting on $Q$ by the restriction of the action on $V$, and trivially on $\gm$, then the action of $G_{n}$ commutes with the action of $\mmu_{2}$ described above, so we get an action of $\mmu_{2} \times G_{n}$, and we have $[B/G_{n}] = [(\gm \times Q)/(\mmu_{2} \times G_{n})]$. On the other hand we have
   \begin{align*}
   \{[(\gm \times Q)/(\mmu_{2} \times G_{n})]\} &=
   \{[(\AA^{1}\times Q)/(\mmu_{2} \times G_{n})]\} -
   \{[Q/(\mmu_{2} \times G_{n})]\}\\
   &= (\LL - 1)\{[Q/(\mmu_{2} \times G_{n})]\}\,.
   \end{align*}
The action of $G_{n}$ on $Q$ is transitive; the stabilizer of a point $p \in Q$ under the action of $\O_{n}$ is isomorphic to $\O_{n-1}$ (see the discussion in \cite[Section 4]{molina-vistoli-classifying}). Call $i\colon \O_{n-1} \subseteq \O_{n}$ the embedding of the stabilizer of $p$; the stabilizer of $p$ under the action of $\mmu_{2} \times \O_{n}$ is the image of $\mmu_{2} \times \O_{n-1}$ under the embedding $j\colon \mmu_{2} \times \O_{n-1} \subseteq \mmu_{2} \times \O_{n}$ defined by $(\alpha, M) \arrto \bigl(\alpha, \alpha i(M)\bigr)$. The stabilizer of $p$ under the action of $\mmu_{2} \times \SO_{n}$ is the inverse image  $j^{-1}(\mmu_{2} \times \SO_{n})$. If $n$ is even, then $j^{-1}(\mmu_{2} \times \SO_{n}) = \mmu_{2} \times \SO_{n-1}$. If $n$ is odd, then $j^{-1}(\mmu_{2} \times \SO_{n})$ is the set of pairs $(\alpha, M) \in \mmu_{2} \times \O_{n-1}$ with $\alpha = \det M$; this is clearly isomorphic to $\O_{n-1}$ via the projection on the second factor. So we have
   \[
   [Q/(\mmu_{2} \times \O_{n})] = [\cB(\mmu_{2} \times \O_{n-1})]\,,
   \]
and furthermore
   \[
   [Q/(\mmu_{2} \times \SO_{n})] = [\cB(\mmu_{2} \times \SO_{n-1})]
   \]
if $n$ is even, while
   \[
   [Q/(\mmu_{2} \times \SO_{n})]  =[\cB\O_{n-1}]
   \]
if $n$ is odd. 

So 
we obtain
   \[
   \{[B/\O_{n}]\} = (\LL - 1)\{\cB\O_{n-1}\} \,,
   \]
while
   \[
   \{[B/\SO_{n}]\} = (\LL - 1)\{\cB\SO_{n-1}\}
   \]
if $n$ is even, and
   \[
   \{[B/\SO_{n}]\} = (\LL - 1)\{\cB\O_{n-1}\}
   \]
if $n$ is odd.

Putting everything together, in the case of $\O_{n}$ we get
   \[
   \{\cB\O_{n}\} = (\LL^{n} - 1)^{-1}\bigl((\LL - 1)\{\cB\O_{n-1}\} + \LL^{-n+2}\{\cB\O_{n-2}\}\bigr)\,;
   \]
a simple calculation shows that the formulas for $\{\cB\O_n\}$ given in the statement verify this recursion (starting from the base values $\{\cB\O_0\} = \{\cB\O_1\}= 1$). 

If $n=2m+1$, then $\O_{n} = \mmu_{2} \times \SO_{n}$, hence

 \[
   \{\cB\SO_{n}\} = \{\cB\O_{n}\} = \LL^{-m^{2}} \prod_{i = 1}^{m}(\LL^{2i} - 1)^{-1}\,.
   \]
If $n$ is even, then we obtain the relation 
   \[
   \{\cB\SO_{n}\} = (\LL^{n} - 1)^{-1}\bigl((\LL - 1)\{\cB\SO_{n-1}\} +
      \LL^{-n+2}\{\cB\SO_{n-2}\}\bigr)\,.
   \]
By the formula for $\{\cB\SO_{n}\}$ for $n$ odd and a simple induction with the recurrence in the last line and the starting case $\{\cB \SO_{2}\} = (\LL - 1)^{-1}$, one can verify that indeed
  \begin{align*}
   \{\cB\SO_{n}\} & =  \LL^{-m^2 + m}(\LL^m -1)^{-1}
   \prod_{i=1}^{m-1} (\LL^{2i}-1)^{-1}\\
  & =   (1+\LL^{-m})\{\cB\O_n\}
   \end{align*}
where $n=2m$.

This completes the proof of Theorem~\ref{thm:main}.
\end{proof}

\begin{corollary}
If $q$ is a non-degenerate split quadratic form, then
   \[
   \{\cB\SO(q)\} = \{\SO(q)\}^{-1}\,.
   \]
\end{corollary}

\begin{proof}
The formula for the class $\{G\} \in \kstack $ of a split connected semisimple group given in \cite[Proposition 2.1]{behrend-dhillon} gives that if $n=2m+1$, then
  \[
   \{\SO_{n}\} = \LL^{2m^2+m}\prod_{i=1}^{m}(1-\LL^{-2i})\, ,
   \]
while if $n=2m$
  \[
   \{\SO_{n}\} = \LL^{2m^2-m}(1-\LL^{-m})\prod_{i=1}^{m-1}(1-\LL^{-2i}) \, .
   \]
Easy algebraic manipulations show that in both cases $\{\cB\SO_n\}=\{\SO_n\}^{-1}$.

(Actually, in \cite{behrend-dhillon} the result is only claimed for the class $\{\SO_{n}\}$ in $\kcomp$, but the proof shows that the formula does in fact hold in $\kstack$.)
\end{proof}

\begin{remark}
We do not know whether the result remains true in characteristic~$2$, or when $q$ is not split. The proof of the formula of Theorem~\ref{thm:main} does use the fact that $q$ is split for even $n$, and probably \cite[Proposition 2.1]{behrend-dhillon} does not work for non-split reductive groups. 
\end{remark}

\bibliographystyle{alpha}
\bibliography{SOn}

\end{document}